\documentclass[10pt]{amsart}

\usepackage{amssymb,amsthm,amstext,amsfonts,amsmath}
\usepackage[spanish,english]{babel}

\newtheoremstyle{theorem}
     {11pt}
     {11pt}
     {}
     {}
     {\bfseries}
     {}
     {.5em}
     {\noindent\thmnumber{#2}. \thmname{#1}\thmnote{#3}}

\theoremstyle{theorem}

\newtheorem{thm}{Theorem}[section]
\newtheorem{lemma}[thm]{Lemma}
\newtheorem{propo}[thm]{Proposition}
\newtheorem{remark}[thm]{Remark}

\newtheorem{coro}[thm]{Corollary}
\newtheorem{ques}[thm]{Question}

\newtheorem{ex}[thm]{Example}
\newtheorem{defi}[thm]{Definition}

\newenvironment{exex}
{\noindent\!\!}{\qed\newline}

\newtheoremstyle{acknowledgement}
     {11pt}
     {11pt}
     {}
     {}
     {\bfseries}
     {}
     {.5em}
     {\noindent \thmname{#1}}

\theoremstyle{acknowledgement}
\newtheorem*{acknowledgement}{Acknowledgement}

\newcommand{\R}[1]{\varrho(#1)}
\newcommand{\cl}[2][X]{\mathrm{cl}_{#1}\!\left(#2\right)}
\newcommand{\cs}{{}\sp\omega2}
\newcommand{\Q}{\mathbb{Q}}
\newcommand{\baire}[1]{{}\sp\omega{#1}}
\newcommand{\cont}{\mathfrak{c}}
\newcommand{\Ex}{\mathrm{Ex}}
\newcommand{\s}{\mathcal{S}}
\newcommand{\B}{\mathcal{B}}
\newcommand{\C}{\mathcal{C}}
\newcommand{\st}[1]{\mathrm{st}(#1)}
\newcommand{\bd}[2][X]{\mathrm{bd}_{#1}\!\left(#2\right)}
\newcommand{\U}{\mathcal{U}}

\title{Spaces of Remote Points}

\author[Hern\'andez-Guti\'errez]{Rodrigo Hern\'andez-Guti\'errez}
	\email[Hern\'andez-Guti\'errez]{rod@matmor.unam.mx}

\author[Hru\v s\'ak]{Michael Hru\v s\'ak}
	\email[Hru\v s\'ak]{michael@matmor.unam.mx}

\author[Tamariz-Mascar\'ua]{Angel Tamariz-Mascar\'ua}
	\email[Tamariz-Mascar\'ua]{atamariz@unam.mx}

\address[Hern\'andez-Guti\'errez and Hru\v s\'ak]{Centro de Ciencias Matem\'aticas, UNAM, A.P. 61-3, Xangari, Morelia, Michoac\'an, 58089, M\'exico}
\address[Tamariz-Mascar\'ua]{Departamento de Matem\'aticas, Facultad de Ciencias, Universidad Nacional Aut\'onoma de M\'exico,
Ciudad Universitaria, M\'{e}xico D.F., 04510, M\'exico}

\thanks{This paper is part of the first author's doctoral dissertation. Research was supported by CONACyT scholarship for Doctoral Students.}

\date{\today}

\begin{document}

\begin{abstract}
Given a Tychonoff space $X$, let $\R{X}$ be the set of remote points of $X$. We view $\R{X}$ as a topological space. In this paper we assume that $X$ is metrizable and ask for conditions on $Y$ so that $\R{X}$ is homeomorphic to $\R{Y}$. This question has been studied before by R. G. Woods and C. Gates. We give some results of the following type: if $X$ has topological property $\mathbf{P}$ and $\R{X}$ is homeomorphic to $\R{Y}$, then $Y$ also has $\mathbf{P}$. We also characterize the remote points of the rationals and irrationals up to some restrictions. Further, we show that $\R{X}$ and $\R{Y}$ have open dense homeomorphic subspaces if $X$ and $Y$ are both nowhere locally compact, completely metrizable and share the same cellular type, a cardinal invariant we define.
\end{abstract}

\maketitle

\section{Introduction}

Given a Tychonoff space $X$, let $\beta X$ denote the \v Cech-Stone compactification of $X$ and $X\sp\ast=\beta X-X$. A point $p\in X\sp\ast$ is called a \emph{remote point} provided $p\notin\cl[\beta X]{A}$ for each nowhere dense subset $A$ of $X$. Following van Douwen \cite{vd51}, we will denote the set of remote points of $X$ by $\R{X}$. In some informal sense, points of $X\sp\ast$ are ``infinite points'' of $X$ and points in $\R{X}$ are ``more infinite than [all] others'' (\cite[1.3]{vd51}). 

The major problem concerning remote points is perhaps their existence. However, in this paper we would like to address another problem that has been forgotten for some years. Our main problem is, in general terms, the following

\begin{quote}
$(\ast)$ Given a Tychonoff space $X$, find all $Y$ such that $\R{X}$ is homeomorphic to $\R{Y}$.
\end{quote}

Notice that problems of type $(\ast)$ can be formulated every time we can construct a space in terms of some other in a topological way; examples of this are compactifications, rings of continuous functions, hyperspaces, absolutes, etcetera. Problem $(\ast)$ has its origin in a paper of R. G. Woods \cite{woodsremote} where the following is stated. (Recall crowded means ``with no isolated points''.)

\begin{thm}\cite{woodsremote}\label{thmwoods}
Let $X$ be a non-compact, locally compact and crowded metrizable space of weight $\kappa$. Then $\R{X}$ is homeomorphic to $\R{\kappa\times\cs}$.
\end{thm}

As noted by C. Gates in \cite{gates}, the proof given in \cite{woodsremote} depends on CH but can be easily modified (in particular, using Proposition \ref{propogates} below) to give a proof in ZFC. We also have the following results by Gates and van Douwen.

\begin{thm}\cite[Corollary 5.8]{gates}\label{thmgates}
Let $X$ be a non-compact, separable and crowded metrizable space whose set of non-locally compact points is compact. Then $\R{X}$ is homeomorphic to $\omega\times\R{\omega\times\cs}$.
\end{thm}

\begin{thm}\cite[Theorem 16.2]{vd51}\label{thmvandouwen}
$\R{\Q}$ and $\R{\baire{\omega}}$ are not homeomorphic because $\R{\Q}$ is a Baire space and $\R{\baire\omega}$ is meager.
\end{thm}

However, Theorems \ref{thmwoods}, \ref{thmgates} and \ref{thmvandouwen} are the only known results concerning $(\ast)$. In this paper, we study the following refinement of $(\ast)$.

\begin{ques}\label{main}
Let $X$ be a metrizable non-compact space. Find some simple or known topological property $\mathbf{P}$ such that if $Y$ is metrizable then $Y$ has $\mathbf{P}$ if and only if $\R{X}$ is homeomorphic to $\R{Y}$.
\end{ques}

The reason we restrict $X$ to be metrizable is because we already know that we have a rich collection of remote points (Proposition \ref{structuremetrizable}). This will allow us to transfer some properties of $X$ to $\R{X}$. However, by Proposition \ref{ponomarevpropo}, we may also consider paracompact $M$-spaces, see Section \ref{paracompact}.

In Section \ref{known} we will give a summary of the known results that will help us attack Question \ref{main}. Section \ref{paracompact} talks about some aspects relating paracompact $M$-spaces to metrizable spaces in this context. Sections \ref{dimlocomp}, \ref{complmetr}, \ref{games} are the main body of the paper where our results on some classes of spaces are proved. 

Some of our results are of the following type: if $X$ has topological property $\mathbf{P}$ and $\R{X}$ is homeomorphic to $\R{Y}$, then $Y$ also has $\mathbf{P}$. In particular we study properties such as dimension, local compactness, topological completeness and $\sigma$-compactness. For the other implication, our main results are perhaps Theorem \ref{remotebaire} and Corollary \ref{remoteQ} that characterize remote points of the irrationals and rationals up to some restrictions. However the main question that remains unanswered in this paper is the following. 

\begin{ques}\label{quesgeneral}
Find all metrizable $X$ such that $\R{X}$ is homeomorphic to either $\R{\Q}$ or $\R{\baire{\omega}}$.
\end{ques}

See Questions \ref{quescomplete} and \ref{quesrationals} for reformulations of Question \ref{quesgeneral}. In Section 5 we also give a classification of nowhere locally compact, completely metrizable spaces by a sort of cardinal invariant we call cellular type. It turns out that cellular type almost characterizes remote points for this class of spaces, see Corollary \ref{corocelltype} and Example \ref{excelltype}.

\section{Existing Tools}\label{known}

Undefined notions can be found in \cite{eng}. Everything else will be defined as soon as it is necessary.

A useful basis for the topology of $\beta X$ is the one formed by the sets of the form $\Ex(U)=\beta X-\cl[\beta X]{X-U}$ where $U\subset X$ is open \cite[Section 3]{vd51}. It is easy to prove that $\Ex(U)\cap X=U$ and $\cl[\beta X]{U}=\cl[\beta X]{\Ex(U)}$. Recall the \v Cech-Stone compactification behaves functorially: every continuous function between Tychonoff spaces $f:X\to Y$ extends to a continuous function $\beta f:\beta X\to \beta Y$.

Let us begin by recalling where we can find remote points in metrizable spaces:

\begin{propo}\label{structuremetrizable}
Let $X$ be a metrizable, non-compact space. Then,
\begin{itemize}
\item[$(a)$] $X$ is nearly realcompact, that is, $\beta X-\upsilon X$ is dense in $X\sp\ast$,
\item[$(b)$] if $G$ is a subset of $\beta X$ of type $G_\delta$ and $\emptyset\neq G\subset X\sp\ast$, then $|G\cap\R{X}|=2\sp\cont$,
\item[$(c)$] $\R{X}$ is dense in $X\sp\ast$.
\end{itemize}
\end{propo}
\begin{proof}
Condition $(a)$ has been observed before in \cite{peters} but for the sake of completeness we sketch a proof here. Consider a basic open subset $\Ex(U)$ of $\beta X$ that intersects $X\sp\ast$. Since $\cl{U}$ is not compact and $X$ is metrizable, there exists a closed discrete infinite subset $\{x_n:n<\omega\}\subset\cl{U}$. We can also assume that $\{x_n:n<\omega\}\subset U$. Now, $X$ is normal so there exists a continuous function $f:X\to\mathbb{R}$ with $X-U\subset f\sp\leftarrow[0]$ and $f(x_n)=n$ for each $n<\omega$. Then it is easy to see that $\beta f\sp\leftarrow[\mathbb{R}\sp\ast]$ is a non-empty set of type $G_\delta$ contained in $\Ex(U)\cap X\sp\ast$. The proof of $(b)$ is implicit in \cite{chaesmith} and explicit in \cite{vd51} for the separable case. Clearly $(c)$ follows from $(a)$ and $(b)$.
\end{proof}

 A function between topological spaces $f:X\to Y$ is called irreducible if it is closed\footnote{Other authors do not require irreducible functions to be closed.} and every time $C$ is a closed subset of $X$, then $f[C]=Y$ if and only if $C=X$ (so in particular it is onto). 

\begin{propo}\cite[Theorem 2.4]{gates}\label{propogates}
Let $X$ be a Tychonoff space, let $Y$ be a normal space and let $f:X\to Y$ be an irreducible continuous function. Then $\R{X}=\beta f\sp\leftarrow[\R{Y}]$ and $\beta f\restriction_{\R{X}}:\R{X}\to\R{Y}$ is a homeomorphism.
\end{propo} 

A space is extremally disconnected if the closure of every open set is also open. Recall that for every regular space $X$ there exists a pair $(EX,k_x)$ where $EX$ is an extremally disconnected regular space and $k_x:EX\to X$ is a perfect and irreducible continuous function (see \cite[Chapter 6]{porterwoods}). It can be proved that $EX$ is unique up to homeomorphism (in the sense of \cite[Theorem 6.7(a)]{porterwoods}). The pair $(EX,k_x)$ is called the \emph{absolute} (or \emph{projective cover}) of $X$. Two regular spaces $X$ and $Y$ are called \emph{coabsolute} if $EX$ is homeomorphic to $EY$. The following results concerning the absolute are relevant to us. The first one is an immediate consequence of the uniqueness of the absolute.

\begin{lemma}\label{perfirriscoabsolute}
If $f:X\to Y$ is a perfect and irreducible continuous function between Tychonoff spaces, then $X$ and $Y$ are coabsolute.
\end{lemma}

\begin{coro}\label{coroabsolute}
If two Tychonoff spaces $X$ and $Y$ are coabsolute, then $\R{X}$ is homeomorphic to $\R{Y}$.
\end{coro}

We state two more technical results we will use.

\begin{lemma}\label{opendense}
Let $X$ be a normal space and $U\subset X$ be open and dense in $X$. Then $\R{X}=\R{U}$.
\end{lemma}
\begin{proof}
Consider the absolute $k_X:EX\to X$. Notice that $k_X\sp\leftarrow[U]$ is an open and dense subset of $EX$, thus, it is extremally disconnected.  It is easy to see that $k_X\!\!\restriction_{k_X\sp\leftarrow[U]}:k_X\sp\leftarrow[U]\to U$ is also an irreducible continuous function so in fact we can identify $k_X\sp\leftarrow[U]$ with $EU$ and $k_X\!\!\restriction_{k_X\sp\leftarrow[U]}$ with $k_U$. By \cite[6.2.C]{porterwoods}, $EU$ is a $C\sp\ast$-embedded subset of $EX$. This implies that $\cl[\beta EX]{EU}$ can be identified with $\beta EU$. No remote point of $EX$ can lie in $\beta EX-\beta EU$ because this set is contained in the closure of the nowhere dense subset $EX-EU$ of $EX$. From this it can be shown that $\R{EX}=\R{EU}$. By applying Proposition \ref{propogates} twice it follows that $\R{U}$ and $\R{X}$ are homeomorphic.
\end{proof}

\begin{propo}\label{subspaces}
Let $X$ be a normal space. Assume $Y$ is a regular closed subset of $X$ and identify $\cl[\beta X]{Y}$ with $\beta Y$. Then $\R{Y}=\R{X}\cap\cl[\beta X]{Y}$. Moreover, since $\cl{X-Y}$ is also a regular closed subset, we can write $\R{X}=\R{Y}\cup\R{\cl{X-Y}}$ and $\R{Y},\R{\cl{X-Y}}$ are disjoint clopen subsets of $\R{X}$.
\end{propo}
\begin{proof}
Define $Z$ to be the direct sum of $Y$ and $\cl{X\setminus Y}$; formally, $Z=A_0\cup A_1$, where $A_0=Y\times\{0\}$ and $A_1=\cl{X\setminus Y}\times\{1\}$. Let $\phi:Z\to X$ be the natural projection to the first coordinate. It is easy to see that $\phi$ is a perfect and irreducible continuous function so we may apply Proposition \ref{propogates} to obtain that $\beta\phi\!\!\restriction_{\R{Z}}:\R{Z}\to\R{X}$ is a homeomorphism. Notice that $A_0$ and $A_1$ are complementary clopen subsets of $Z$. Thus, $\cl[\beta Z]{A_i}$ can be identified with $\beta A_i$ for $i\in2$ and $\beta A_0\cap\beta A_1=\emptyset$. From this it is straightforward that $\R{A_i}=\R{X}\cap\cl[\beta Z]{A_i}$ for $i\in 2$, $\R{Z}=\R{A}\cup\R{B}$ and $\R{A}\cap\R{B}=\emptyset$. From the fact that both $\phi\!\!\restriction_{A_0}:A_0\to Y$ and $\phi\!\!\restriction_{A_1}:A_1\to\cl{X\setminus Y}$ are homeomorphisms it is not hard to prove that $\beta\phi[\R{A_0}]=\R{Y}$ and $\beta\phi[\R{A_1}]=\R{\cl{X\setminus Y}}$. The result follows from these observations.
\end{proof}

\section{Paracompact $M$-spaces}\label{paracompact}

To make clear how far we can get by considering metrizable spaces only, we quote the following result. Recall that a space $X$ is an $M$-space if there exists a sequence $\{\C_n:n<\omega\}$ of covers of $X$ such that
\begin{itemize}
\item[$(i)$] if $x_n\in\st{x,\C_n}$ for each $n<\omega$, then $\{x_n:n<\omega\}$ has a cluster point,
\item[$(ii)$] for each $n<\omega$, $\C_{n+1}$ star-refines $\C_n$.
\end{itemize}

\begin{propo}\cite{poncoabsolutes}\label{ponomarevpropo}
Let $X$ be a Tychonoff space. Then the following are equivalent
\begin{itemize}
\item[$(a)$] $X$ is coabsolute with a metrizable space,
\item[$(b)$] there exists a metrizable space that is a perfect and irreducible continuous image of $X$,
\item[$(c)$] $X$ is a paracompact $M$-space with a $\sigma$-locally finite $\pi$-base.
\end{itemize}
\end{propo}

Actually, in \cite{poncoabsolutes}, $(c)$ says ``paracompact $p$-space'' (in the sense of Arkhangel'ski\u\i) but this is equivalent to the formulation we have given (see \cite{handbookgenmetric} for details).

So according to Propositions \ref{propogates} and \ref{ponomarevpropo}, we will be able to obtain results about remote points of paracompact $M$-spaces. However, our results will be stated in terms of metrizable spaces only. The main reason for doing this is that for our arguments it is enough to consider metrizable spaces. Moreover, the corresponding results for paracompact $M$-spaces can be easily obtained by using Proposition \ref{ponomarevpropo}.

\begin{remark}
By Proposition \ref{ponomarevpropo}, it is not hard to prove that in Proposition \ref{coabsolute}, Theorems \ref{remotebaire} and \ref{coabsoluteQ} and Corollary \ref{remoteQ} we can replace ``(completely) metrizable'' by ``(\v Cech-complete) paracompact $M$-space with a $\sigma$-locally finite $\pi$-base''.
\end{remark}

We will now address another technical matter. Notice that Proposition \ref{propogates} talks about irreducible closed mappings, while talking about coabsolutes we asked that the function be perfect. We now show that there is nothing else we can obtain using Proposition \ref{propogates} inside the class of paracompact $M$-spaces.

\begin{propo}\label{irrisperfect}
Let $f:X\to Y$ be an irreducible continuous function between paracompact $M$-spaces. Then $f$ is perfect.
\end{propo}
\begin{proof}
By $(b)$ in Proposition \ref{ponomarevpropo}, there exists a metrizable space $M$ and a perfect and irreducible continuous function $g:Y\to M$. We will follow the proof of Va\u\i n\v ste\u\i n's Lemma and the Hanai-Morita-Stone Theorem from \cite[4.4.16]{eng}. Let $h=g\circ f$.

\vskip11pt
\noindent{\it Claim:} For every $p\in M$, $\bd{h\sp\leftarrow(p)}$ is countably compact.
\vskip11pt

To prove the Claim, let $\{x_n:n<\omega\}\subset \bd{h\sp\leftarrow(p)}$. Since $M$ is metrizable let $\{U_n:n<\omega\}$ be a local open basis of $p$ such that $\cl[M]{U_{n+1}}\subset U_n$ for each $n<\omega$. Let $\{\C_n:n<\omega\}$ be the sequence of open covers for $X$ given by the definition of an $M$-space. For each $n<\omega$, let $y_n\in(h\sp\leftarrow[U_n]-h\sp\leftarrow(p))\cap\st{x_n,\C_n}$. Then it is easy to see that $\{h(y_n):n<\omega\}$ is a sequence converging to $p$. Since $h$ is a closed function, it follows that there exists a cluster point $q\in h\sp\leftarrow(p)$ of $\{y_n:n<\omega\}$.

We construct a strictly increasing function $\phi:\omega\to\omega$ with $\phi(n)\geq n+1$ for all $n<\omega$ as follows: for each $n<\omega$ let $\phi(n)$ be such that $y_{\phi(n)}\in\st{q,\C_{n+1}}$. Since $y_{\phi(n)}\in\st{x_{\phi(n)},\C_{\phi(n)}}$ and $\phi(n)\geq n+1$, by condition $(ii)$ in the definition of an $M$-space we obtain that $x_{\phi(n)}\in\st{q,\C_n}$. Thus, by condition $(i)$ in the definition of an $M$-space we obtain that $\{x_{\phi(n)}:n<\omega\}$ has a cluster point. Such cluster point must be in $\bd{h\sp\leftarrow(p)}$ so the Claim follows.

Now, by the Claim and the fact that $X$ is paracompact we have that $\bd{h\sp\leftarrow(p)}$ is compact for each $p\in M$. Since $h$ is irreducible, either $\bd{h\sp\leftarrow(p)}=h\sp\leftarrow(p)$ or $h\sp\leftarrow(p)$ is a singleton. Thus, $h$ is perfect and so is $f$.
\end{proof}

\section{Dimension and Local Compactness}\label{dimlocomp}

We start by showing that it is enough to consider \emph{strongly $0$-dimensional} spaces. Recall that a Tychonoff space is strongly $0$-dimensional if every two disjoint zero subsets can be separated by a clopen subset.

\begin{propo}\label{dimension}
Each metrizable space is coabsolute with a strongly $0$-dimensional metrizable space.
\end{propo}
\begin{proof}
Let $X$ be a metrizable space. By a result of Morita (\cite{morita}, see \cite[4.4.J]{eng}) there is some infinite cardinal $\kappa$, a space $Z\subset\baire{\kappa}$ and a perfect continuous surjection $g:Z\to X$. Using the Kuratowski-Zorn Lemma it is easy to find a closed subset $Y\subset Z$ minimal with the property that $g[Y]=X$. Then $f=g\restriction_Y:Y\to X$ is easily seen to be a perfect and irreducible continuous function. Notice that $Y$ is also strongly $0$-dimensional because it is a subspace of $\baire{\kappa}$ (\cite[7.3.4]{eng}). Then $X$ and $Y$ are coabsolute by Lemma \ref{perfirriscoabsolute}.
\end{proof}

Thus, to study the remote points of a metrizable space $X$ it is enough to assume that $X$ is strongly $0$-dimensional by Propositions \ref{propogates} and \ref{dimension}.

The next step is to see that local compacteness is distinguished by remote points. For any space $X$, let $LX$ be the points where $X$ is locally compact and $NX=X-\cl{LX}$. Gates \cite{gates} has already noticed the following Corollary of Proposition \ref{subspaces}.

\begin{lemma}\label{loccomplemma}
Let $X$ be a normal space. Then $\R{X}$ is homeomorphic to the direct sum $\R{\cl{LX}}\oplus\R{\cl{NX}}$.
\end{lemma}

Now we must find a topological way to distinguish between $\R{\cl{LX}}$ and $\R{\cl{NX}}$. A point $p$ in a space $X$ is called a \emph{$\kappa$-point} of $X$, where $\kappa$ is a cardinal, if there is a collection $\mathcal{S}$ of pairwise disjoint open subsets of $X$ such that $|\mathcal{S}|\geq\kappa$ and $p\in\cl{U}$ for each $U\in\mathcal{S}$. If $p$ is not a $2$-point of $X$ it is said that $X$ is \emph{extremally disconnected at} $p$. Notice that a space is extremally disconnected if and only if it is extremally disconnected at each of its points.

\begin{propo}\cite[Theorem 5.2]{vd82}\label{cpointsvd}
If $X$ is a realcompact, locally compact and non-compact Tychonoff space, then each point of $X\sp\ast$ is a $\cont$-point of $X\sp\ast$.
\end{propo}

The following lemma is folklore.

\begin{lemma}\label{denseed}
Let $D$ be a dense subset of a space $X$. If $X$ is extremally disconnected at each point of $D$, then $D$ is extremally disconnected.
\end{lemma}

\begin{thm}\label{loccomp}
Let $X$ be a metrizable space. Then $\R{\cl{LX}}$, if non-empty, contains a dense set of $\cont$-points and $\R{NX}$ is extremally disconnected.
\end{thm}
\begin{proof}
Let us start with $NX$. By \cite[Corollary 5.2]{vd51}, $\beta(NX)$ is extremally disconnected at each point of $\R{NX}$. Notice that $(NX)\sp\ast$ is dense in $\beta(NX)$ so by $(c)$ in Proposition \ref{structuremetrizable}, $\R{NX}$ is dense in $\beta(NX)$.  It follows from Lemma \ref{denseed} that $\R{NX}$ is extremally disconnected.

Let $Y=\cl{LX}$ and fix some metric for $Y$. Assume that $Y$ is not compact so that $\R{Y}\neq\emptyset$. Let $\Ex(U)$ be a basic open subset of $\beta Y$ that intersects $Y\sp\ast$.

Since $\cl[Y]{U}$ is not compact, we may find a closed discrete infinite set $\{y_n:n<\omega\}\subset U$, just as in the proof of $(a)$ in Proposition \ref{structuremetrizable}. For each $n<\omega$, let $U_n$ be an open subset of $Y$ such that $y_n\in U_n$, $\cl[Y]{U_n}\subset U$ and $\cl[Y]{U_n}$ is a compact subset of diameter at most $\frac{1}{n+1}$.

Let $A=\bigcup\{\cl[Y]{U_n}:n<\omega\}$. Then $A$ is a regular closed subset of $Y$ and since $Y$ is a normal space, $\cl[\beta Y]{A}\subset \Ex(U)$. By Proposition \ref{subspaces} we have that $\R{A}$ is a clopen subspace of $\R{Y}\cap \Ex(U)$. Notice that $A$ is realcompact (each $\cl{U_n}$ has countable weight) so by Proposition \ref{cpointsvd} each point of $\R{A}$ is a $\cont$-point of $\cl[Y]{A}-A(=A\sp\ast)$. Using $(c)$ in Proposition \ref{structuremetrizable} it follows that each point of $\R{A}$ is a $\cont$-point of $\R{A}$. Thus, any point of $\R{A}$ is a $\cont$-point of $\R{Y}$ contained in $\Ex(U)$.

\end{proof}

The following follows immediately from Lemma \ref{opendense}, Lemma \ref{loccomplemma} and Theorem \ref{loccomp}.

\begin{coro}\label{coroloccomp}
Let $X$ and $Y$ be metrizable spaces. If $h:\R{X}\to\R{Y}$ is a homeomorphism then $h[\R{\cl{LX}}]=\R{\cl{LY}}$ and $h[\R{NX}]=\R{NY}$.
\end{coro}

Results for $LX$ were mentioned in Theorems \ref{thmwoods} and \ref{thmgates}.  We will now direct our efforts towards nowhere locally compact spaces (that is, spaces where $X=NX$). If $X$ is nowhere locally compact (and metrizable) and $\R{X}$ is homeomorphic to $\R{Y}$ then it follows from Corollary \ref{coroloccomp} that $\cl{LX}$ is compact. Since $\R{Y}=\R{NY}$ in this case, we may restrict to the case when both $X$ and $Y$ are nowhere locally compact.

\section{Completely Metrizable Spaces}\label{complmetr}

We would like to give a classification of the set of remote points of nowhere locally compact, completely metrizable spaces in the spirit of Theorem \ref{thmwoods}. Recall that a metrizable space is completely metrizable if and only if it is \v Cech-complete (\cite[Theorem 4.3.26]{eng}). 

We shall start by characterizing spaces coabsolute with the Baire space $\baire{\kappa}$, where $\kappa$ is an infinite cardinal, see Proposition \ref{coabsolute}. The proof of the following Lemma is easy.

\begin{lemma}\label{lemmacechcomplete}
Let $f:X\to Y$ be an irreducible and perfect continuous function between Tychonoff spaces. Then
\begin{itemize}
\item[(a)] $X$ is a \v Cech-complete space if and only if $Y$ is and
\item[(b)] $X$ is nowhere locally compact if and only if $Y$ is.
\end{itemize}
\end{lemma}

Recall that for a space $X$, $c(X)$ denotes its cellularity. We will say that $X$ is \emph{of uniform cellularity} ($\kappa$) if $c(X)=c(U)(=\kappa)$ for each non-empty open subset $U\subset X$. We are interested in spaces of uniform cellularity because $\baire{\kappa}$ is such a space and this property is preserved as the following result shows. If $f:X\to Y$ is a function and $A\subset X$, let $f\sp\sharp[A]=Y-f[X-A]$.

\begin{lemma}\label{lemmaunifcell}
$\empty$\vskip11pt
\begin{itemize}
\item[$(i)$] Let $f:Y\to X$ be an irreducible continuous function. Then $c(X)=c(Y)$ and $X$ is of uniform cellularity if and only if $Y$ is.
\item[$(ii)$] Let $X$ be a space and $D\subset X$ a dense subset. Then $c(X)=c(D)$ and $X$ is of uniform cellularity if and only if $D$ is.
\end{itemize}
\end{lemma}
\begin{proof}
We start with $(i)$. That $c(X)=c(Y)$ is easy to prove. Assume that $X$ is of uniform cellularity and let $U\subset Y$ be a non-empty open subset. Let $V=f\sp\sharp[U]$. Then it is easy to see that $f\restriction_{f\sp\leftarrow[V]}:f\sp\leftarrow[V]\to V$ is an irreducible and continuous function. Thus, we already know that $c(V)=c(f\sp\leftarrow[V])$. Since $f\sp\leftarrow[V]\subset U$ we obtain that $c(X)=c(V)=c(f\sp\leftarrow[V])\leq c(U)\leq c(X)$ so $c(U)=c(X)$. The rest of the argument for $(i)$ is similar and the proof of $(ii)$ is straightforward.
\end{proof}

\begin{propo}\label{coabsolute}
Let $X$ be a metrizable space and $\kappa$ an infinite cardinal. Then $X$ is coabsolute with $\baire{\kappa}$ if and only if $X$ is a nowhere locally compact, completely metrizable space of uniform cellularity $\kappa$.
\end{propo}
\begin{proof}
If $X$ is coabsolute with $\baire{\kappa}$ then the result follows from Lemmas \ref{lemmacechcomplete} and \ref{lemmaunifcell}. Now assume that $X$ has the properties given in the Proposition. By Proposition \ref{dimension} we may assume that $X$ is coabsolute with a strongly $0$-dimensional metrizable space. In \cite{alexury} and \cite{stonenonsep} it is proved that a completely metrizable, nowhere locally compact, strongly $0$-dimensional space such that $w(U)=\kappa$ for each non-empty open subset $U$ is homeomorphic to $\baire{\kappa}$ (see \cite[6.2.A and 7.2.G]{eng}). The result now follows.
\end{proof}

Notice that in Proposition \ref{coabsolute}, nowhere locally compact can be omitted when $\kappa>\omega$ and of uniform cellularity $\omega$ simply means being separable.

\begin{thm}\label{remotebaire}
Let $X$ be a nowhere locally compact, completely metrizable space and $\kappa>\omega$. Then $\R{X}$ is homeomorphic to $\R{\baire{\kappa}}$ if and only if $X$ is of uniform cellularity $\kappa$.
\end{thm}
\begin{proof}
If $X$ is of uniform cellularity $\kappa$, use Proposition \ref{coabsolute}. Now, assume that $\R{X}$ is homeomorphic to $\R{\baire{\kappa}}$. Both $X$ and $\R{X}$ are dense in $\beta X$. Also, $\R{\baire{\kappa}}$ and $\baire{\kappa}$ are dense in $\beta(\baire{\kappa})$. By Lemma \ref{lemmaunifcell} we obtain that $X$ is of uniform cellularity $\kappa$.
\end{proof}

Wondering if the hypothesis that $X$ is complete is necessary in Theorem \ref{remotebaire} we observe the following.

\begin{ex}\label{exnometrizable}
For each completely metrizable, realcompact and non-compact space $X$ there exists a Tychonoff space $Y$ that is neither \v Cech complete nor an $M$-space such that $\R{X}$ is homeomorphic to $\R{Y}$.
\end{ex}
\begin{exex}
Let $D=\{x_n:n<\omega\}$ be a countable closed discrete and infinite subset of $X$. Let $p\in\cl[\beta X]{D}-X$ and let $Y=X\cup\{p\}$ as a subspace of $\beta X$. Notice that $\beta X=\beta Y$. We now show that $Y$ is not \v Cech-complete and it is not an $M$-space.

Assume that $Y$ is \v Cech-complete and let $\{U_n:n<\omega\}$ be a family of open subsets of $\beta X$ whose intersection is $Y$. Thus, $\{U_n\cap X\sp\ast:n<\omega\}$ witnesses that $\{p\}$ is a $G_\delta$ set of $X\sp\ast$. But $X$ is realcompact so there exists a subset of type $G_\delta$ of $\beta X$ that contains $p$ and misses $X$. It easily follows that $\{p\}$ is a $G_\delta$ set of $\beta X$. But this contradicts $(b)$ in Proposition \ref{structuremetrizable}.

Now assume that $Y$ is an $M$-space, we will reach a contradiction. Let $\{\C_n:n<\omega\}$ be as in the definition of an $M$-space. Let $d$ be a metric for $X$. As in the proof of Proposition \ref{irrisperfect}, we consider a strictly increasing $\phi:\omega\to\omega$ as follows: let $\phi(n)$ be such that $x_{\phi(n)}\in\st{p,\C_n}$. For each $n<\omega$, let $y_n\in\st{p,\C_n}-D$ be such that $d(x_{\phi(n)},y_n)<\frac{1}{n+1}$. Then it can be proved that $\{y_n:n<\omega\}$ is a closed discrete subset of $X$. Since $X$ is normal, $\{y_n:n<\omega\}$ is also closed in $Y$. But this contradicts $(i)$ in the definition of an $M$-space.

Finally, $p\notin\R{X}$ because $D$ is nowhere dense. Thus, $\R{X}=\R{Y}$.
\end{exex}

However, since the space in Example \ref{exnometrizable} is not an $M$-space, we can ask the following more specific question.

\begin{ques}\label{quescomplete}
Let $X$ and $Y$ be metrizable spaces such that $\R{X}$ is homeomorphic to $\R{Y}$. If $X$ is completely metrizable, must $Y$ be also completely metrizable?
\end{ques}

We have obtained a characterization of spaces with remote points homeomorphic to the remote points of some specific completely metrizable spaces in Theorem \ref{remotebaire}. In an effort to classify them all we make the following definition.

\begin{defi}
Let $X$ be a space, let $\s$ be a set of infinite cardinals and $\phi$ a function with domain $\s$ such that for each $\kappa\in\s$, $\phi(\kappa)$ is a cardinal $\geq\kappa$. We will say that $X$ has \emph{cellular type} $(\s,\phi)$ if there exists a pairwise disjoint family of open subsets $\B=\{V(\kappa,\alpha):\kappa\in\s,\alpha<\phi(\kappa)\}$ of $X$ whose union is dense in $X$ and such that if $\kappa\in\s$ and $\alpha<\phi(\kappa)$ then $V(\kappa,\alpha)$ is of uniform cellularity $\kappa$. In this case we say that $\B$ is a witness to the cellular type of $X$.
\end{defi}

\begin{lemma}
Every crowded metrizable space has a cellular type.
\end{lemma}
\begin{proof}
Let $X$ be a crowded metrizable space. By the Kuratowski-Zorn Lemma, find a maximal pairwise disjoint family $\U$ of open, non-empty subsets of $X$ of uniform cellularity. Notice that $\bigcup\U$ is dense, otherwise let $V$ be an open subset of $X-\cl{\bigcup\U}$ of minimal cellularity; $\U\cup\{V\}$ contradicts the maximality of $\U$. We define
$$
\s=\{\kappa:\kappa\textrm{ is a cardinal and there is }U\in\U\textrm{ such that }c(U)=\kappa\}.
$$
Clearly $\s$ is a set of infinite cardinals because $X$ is crowded. For each $\kappa\in\s$ let
$$
\U(\kappa)=\{U\in\U:c(U)=\kappa\}.
$$
We may assume that $|\U(\kappa)|\geq\kappa$ by the following argument. If the size of this family is smaller than $\kappa$, using the fact that the cellularity is attained in metrizable spaces (\cite[8.1(d)]{hodelhandbook}), replace $\U(\kappa)$ by a family of $\kappa$ pairwise disjoint open subsets of $\bigcup\U(\kappa)$ that is dense in $\bigcup\U(\kappa)$. We finally define $\phi$: for each $\kappa\in\s$ we let $\phi(\kappa)=|\U(\kappa)|$. Clearly $\U$ witnesses that $X$ has cellular type $(\s,\phi)$.
\end{proof}

\begin{lemma}
If a Tychonoff space $X$ has cellular types $(\s,\phi)$ and $(\mathcal{T},\psi)$, then $(\s,\phi)=(\mathcal{T},\psi)$.
\end{lemma}
\begin{proof}
Let $\{V(\kappa,\alpha):\kappa\in\s,\alpha<\phi(\kappa)\}$ witness type $(\s,\phi)$ and $\{W(\kappa,\alpha):\kappa\in\mathcal{T},\alpha<\psi(\kappa)\}$ witness type $(\mathcal{T},\psi)$. For each $\kappa\in\s$, $V(\kappa,0)$ is a non-empty open subset of $X$ so it must intersect some $W(\tau,\alpha)$, with $\tau\in\mathcal{T}$ and $\alpha<\psi(\tau)$. By the definition of uniform cellularity it follows that $\kappa=\tau$ so $\s\subset\mathcal{T}$. By a similar argument $\s=\mathcal{T}$. Notice that $V(\kappa,\alpha)\cap W(\tau,\beta)\neq\emptyset$ implies $\kappa=\tau$. Assume $\phi(\lambda)<\psi(\lambda)$ for some $\lambda\in\s$. For each $\alpha<\phi(\lambda)$ let $J(\alpha)=\{\beta<\psi(\lambda):V(\lambda,\alpha)\cap W(\lambda,\beta)\neq\emptyset\}$. Then $|J(\alpha)|\leq\lambda$ because $c(V(\lambda,\alpha))=\lambda$. Let $\gamma\in\psi(\lambda)-\bigcup\{J(\alpha):\alpha<\phi(\lambda)\}$, then it follows that $W(\lambda,\gamma)$ does not intersect any element of $\{V(\kappa,\alpha):\kappa\in\s,\alpha<\phi(\kappa)\}$, which contradicts the density of $\bigcup\{V(\kappa,\alpha):\kappa\in\s,\alpha<\phi(\kappa)\}$. This completes the proof.
\end{proof}

So at least metrizable spaces have a cellular type and cellular type is unique. Sometimes cellular type can be transfered from one space to another. For example we have the following easy transfer result.

\begin{lemma}
Let $X$ be any space and $D$ a dense subset of $X$. Then $X$ has cellular type $(\s,\phi)$ if and only if $D$ has cellular type $(\s,\phi)$.
\end{lemma}

In the case of nowhere locally compact, completely metrizable spaces, cellular type can be transfered to the remote points. For this, nice witnesses are needed.

\begin{lemma}\label{celltypepibase}
Let $X$ be a regular space with cellular type $(\s,\phi)$. Then there exists a witness family $\B$ of the cellular type of $X$ such that any two different members of $\B$ have disjoint closures.
\end{lemma}
\begin{proof}
Let $\{W(\kappa,\alpha):\kappa\in\s,\alpha<\phi(\kappa)\}$ witness the cellular type of $X$. For each $\kappa\in\s$ and $\alpha<\phi(\kappa)$ let $\B(\kappa,\alpha)$ be a maximal family of open subsets of $W(\kappa,\alpha)$ whose closures are pairwise disjoint and contained in $W(\kappa,\alpha)$. Clearly $|\B(\kappa,\alpha)|\leq\kappa$. Give an enumeration $\{V(\kappa,\alpha):\alpha<\phi(\kappa)\}$ of $\bigcup\{\B(\kappa,\alpha):\alpha<\phi(\kappa)\}$. Clearly $\{V(\kappa,\alpha):\kappa\in\s,\alpha<\phi(\kappa)\}$ is the witness we were looking for.
\end{proof}

\begin{thm}\label{thmcelltype}
Let $X$ be a nowhere locally compact and completely metrizable space of cellular type $(\s,\phi)$. Then there exists a family $\{V(\kappa,\alpha):\kappa\in\s,\alpha<\phi(\kappa)\}$ consisting of clopen subsets of $\R{X}$ that witnesses that $\R{X}$ has cellular type $(\s,\phi)$ and has the additional property that for each $\kappa\in\s$ and $\alpha<\phi(\kappa)$, $V(\kappa,\alpha)$ is homeomorphic to $\R{\baire{\kappa}}$.
\end{thm}

\begin{proof}
Let $\{W(\kappa,\alpha):\kappa\in\s,\alpha<\phi(\kappa)\}$ witness the cellular type of $X$. By Lemma \ref{celltypepibase}, we may assume that every two subsets of this family have disjoint closures. Let $D(\kappa,\alpha)=\cl{W(\kappa,\alpha)}$ for each $\kappa\in\s$ and $\alpha<\phi(\kappa)$, also define 
$$
V(\kappa,\alpha)=\R{X}\cap\cl[\beta X]{D(\kappa,\alpha)}.
$$
It easily follows from Proposition \ref{subspaces} that $V(\kappa,\alpha)=\R{D(\kappa,\alpha)}$ for each $\kappa\in\s$, $\alpha<\phi(\kappa)$ and $\B=\{V(\kappa,\alpha):\kappa\in\s,\alpha<\phi(\kappa)\}$ is a pairwise disjoint family of clopen subsets of $\R{X}$. Since $\R{X}$ is dense in $\beta X$ (Proposition \ref{structuremetrizable}), $V(\kappa,\alpha)\neq\emptyset$ for each $\kappa\in\s$, $\alpha<\phi(\kappa)$ and $\bigcup\B$ is dense.

Finally, fix $\kappa\in\s$ and $\alpha<\phi(\kappa)$, we now prove that $V(\kappa,\alpha)$ is homeomorphic to $\R{\baire{\kappa}}$, which will complete the proof. Notice that $D(\kappa,\alpha)$ is nowhere locally compact, completely metrizable and of uniform cellularity $\kappa$. The result now follows from Theorem \ref{remotebaire}.
\end{proof}

\begin{coro}\label{corocelltype}
If $X$ and $Y$ are nowhere locally compact, completely metrizable spaces with the same cellular type then $\R{X}$ and $\R{Y}$ have open dense homeomorphic subspaces.
\end{coro}

The following is true but maybe not worth proving in detail. We leave it as an exercise to the reader. 

\begin{propo}
$\empty$\vskip11pt
\begin{itemize}
\item Let $X$ and $Y$ be coabsolute Tychonoff spaces. Then $X$ has cellular type $(\s,\phi)$ if and only if $Y$ has cellular type $(\s,\phi)$. 
\item Every paracompact $p$-space has a cellular type.
\end{itemize}
\end{propo}

So the remaining question is if cellular type completely characterizes remote points of nowhere locally compact and completely metrizable spaces. We end this section showing that this is not the case by means of an example.

\begin{ex}\label{excelltype}
There exist two nowhere locally compact and completely metrizable spaces $X$ and $Y$ that have the same cellular type but such that $\R{X}$ is not homeomorphic to $\R{Y}$.
\end{ex}
\begin{exex}
For each $n<\omega$, let $X_n$ be homeomorphic to $\baire{\omega_n}$ such that $\{X_n:n<\omega\}$ are pairwise disjoint. Let us define $X=\oplus\{X_n:n<\omega\}$, clearly this is a nowhere locally compact and completely metrizable space. For each $n<\omega$ let $K_n=\cl[\beta X]{X_n}$ and let $P=\beta X-\bigcup\{K_n:n<\omega\}$.

We now define $T$ as the quotient space of $\beta X$ obtained by identifying $P$ to a point and let $\rho:\beta X\to T$ be this identification. Let $p\in T$ be such that $\{p\}=\rho[P]$ and define $Y=\{p\}\cup(\bigcup\{X_n:n<\omega\})$ as a subspace of $T$.

Notice $X$ and $Y$ have the same cellular type, simply because $X$ is open and dense in $Y$.

To see that $Y$ is metrizable, we may use the Bing-Nagata-Smirnov Theorem \cite[4.4.7]{eng}, since $X$ is already metrizable it is enough to notice that $Y$ is first-countable at $p$.  Since $Y$ is a $G_\delta$ in $T$, $Y$ is completely metrizable. 

We claim that $Y$ is $C\sp\ast$-embedded in $T$. Let $f:Y\to[0,1]$ be a continuous function. Let $g=f\restriction_{X}$, we now prove that $\beta g:\beta X\to [0,1]$ is constant restricted to $P$. Assume this is not the case and let $x,y\in P$ be such that $\beta g(x)<\beta g(y)$. Let $\epsilon=\frac{1}{3}(\beta g(y)-\beta g(x))$ and define $U=\beta g\sp\leftarrow[(-\infty,\beta g(x)+\epsilon)]$ and $V=\beta g\sp\leftarrow[(\beta g(y)-\epsilon,\infty)]$. Then there exist two closed countable discrete subsets $D_0$, $D_1$ of $X$ such that $D_0\subset U$ and $D_1\subset V$. Clearly, both $D_0$ and $D_1$ converge to $p$ in $Y$, this contradicts the definition of $U$ and $V$. Thus the function $F:T\to[0,1]$ given by $F(x)=\beta g(x)$ if $x\neq p$ and $\{F(p)\}=\beta g[P]$ is a continuous extension of $f$. Thus, $T=\beta Y$.

Since $P$ is a $G_\delta$ subset of $\beta X$, by Proposition \ref{structuremetrizable} we have that $\R{X}\cap P\neq\emptyset$. Also notice that
\begin{eqnarray*}
\R{X}=&\R{Y}\cup(\R{X}\cap P),&\textrm{ and}\\
\R{Y}=&\bigcup\{\R{X_n}:n<\omega\}.&
\end{eqnarray*}

To prove that $\R{X}$ is not homeomorphic to $\R{Y}$ it is enough to notice the following two facts which show different topological properties of points in $\R{X}\cap P$ to those in $\R{Y}$.

\begin{itemize}
\item[$(a)$] if $n<\omega$, $c(\R{X_n})=\omega_n$,
\item[$(b)$] if $q\in\R{X}\cap P$ then for every open set $U\subset \R{X}$ such that $q\in U$, $c(U)\geq\omega_\omega$.
\end{itemize}

Statement $(a)$ is clear. For statement $(b)$ let $q\in\R{X}\cap P$ and $U$ be an open subset of $\R{X}$ with $q\in U$. Let $A\subset\omega$ be an infinite set such that $U\cap K_n\neq\emptyset$ for all $n\in A$. For each $n\in A$ we may choose a pairwise disjoint family of open sets $\{V(\alpha,n):n<\omega_n\}$ of $\R{X_n}\cap U$. Then $\{V(\alpha,n):n<\omega,\alpha<\omega_n\}$ is a collection of $\omega_\omega$ pairwise open sets contained in $U$. Thus, $c(U)\geq\omega_\omega$.
\end{exex}

\section{Meager vs Comeager}\label{games}

In this section we consider separable metrizable spaces. We already know some spaces $X$ with $\R{X}$ homeomorphic to $\R{\baire{\omega}}$ (Theorem \ref{remotebaire}). We start by considering the problem for $\Q$.

\begin{lemma}\label{coabsoluteremainder}
Any two remainders of a nowhere locally compact Tychonoff space are coabsolute.
\end{lemma}
\begin{proof}
Let $X$ be nowhere locally compact Tychonoff space. Consider any compactification $T$ of $X$ with remainder $Y$ and let $f:\beta X\to T$ be the continuous extension of the identity function. By \cite[1.8(i)]{porterwoods}, $X\sp\ast=f\sp\leftarrow[Y]$. Then $g=f\restriction_X:X\sp\ast\to Y$ is a perfect and irreducible continuous function. So $Y$ is coabsolute with $X\sp\ast$.
\end{proof}

\begin{thm}\label{coabsoluteQ}
Let $X$ be a metrizable space. Then $X$ is coabsolute with $\Q$ if and only if $X$ is nowhere locally compact and $\sigma$-compact.
\end{thm}
\begin{proof}
If $X$ is coabsolute with $\Q$, then $X$ is nowhere locally compact and $\sigma$-compact because these properties are preserved under perfect and irreducible continuous functions. Assume now that $X$ is nowhere locally compact and $\sigma$-compact. By Proposition \ref{dimension} we may assume that $X$ is strongly $0$-dimensional. Since $X$ is $\sigma$-compact, it is separable so it can be embedded in $\cs$. Moreover $X$ is crowded so the closure of $X$ in $\cs$ is also crowded. So we may assume that $X$ is dense in $\cs$, let $Y=\cs-X$. But then $Y$ is a separable, nowhere locally compact and completely metrizable $0$-dimensional space. Then $Y$ is homeomorphic to $\baire{\omega}$ (\cite{alexury}). Notice that $\cs$ is a compactification of $Y$ and there is a compactification of $\baire{\omega}$ with remainder $\Q$ ($\cs$ where $\Q$ is taken as the set of eventually constant functions). By Lemma \ref{coabsoluteremainder} we obtain that $X$ is coabsolute with $\Q$.
\end{proof}

Another proof of Theorem \ref{coabsoluteQ} can be given by considering the following result of van Mill and Woods.

\begin{thm}\cite[Theorem 3.1]{vmwoods}
Let $X$ be a $\sigma$-compact, nowhere locally compact metrizable space. Then there exists a perfect and irreducible continuous function $f:\Q\times\cs\to X$.
\end{thm}

\begin{coro}\label{remoteQ}
If $X$ is a $\sigma$-compact, nowhere locally compact, metrizable space then $\R{X}$ is homeomorphic to $\R{\Q}$.
\end{coro}

Notice we have the following situation: For two specific spaces $\baire{\omega}$ and $\Q$ we have found non-trivial classes of metrizable spaces that have the same set of remote points as these spaces. Now we want to know if these classes of spaces are the best possible. We were not able to solve this problem but we will prove Proposition \ref{thmgames} that is in the spirit of van Douwen's Theorem \ref{thmvandouwen}.

For a topological space $X$ we consider the \emph{Choquet game} with two players $I$ and $II$ in $\omega$ steps as follows. In step $n<\omega$, first $I$ chooses a non-empty open subset $U_n$ of $X$ with $U_n\subset V_{n-1}$ if $n\neq 0$ and then $II$ chooses a non-empty open subset $V_n$ of $U_n$. We say that player $II$ wins if $\bigcap V_n\neq\emptyset$, otherwise player $I$ wins. Call $X$ a \emph{Choquet space} if player $II$ has a winning strategy in the Choquet game for $X$. See \cite[8C]{kechris} for details.

\begin{lemma}\cite[8.17]{kechris}\label{choquetcomeager}
A separable metric space $(X,d)$ is a Choquet space if and only if $(X,d)$ is comeager in its completion.
\end{lemma}

Thus, every separable completely metrizable space is a Choquet space. Also notice that a $\sigma$-compact nowhere locally compact metrizable space is meager. So in some sense $\baire{\omega}$ and $\Q$ are dual.

\begin{lemma}\label{gamecoabsolute}
Let $f:X\to Y$ be an irreducible continuous function between crowded regular spaces. Then
\begin{itemize}
\item[$(a)$] $X$ is meager if and only if $Y$ is meager, and
\item[$(b)$] $X$ is a Choquet space if and only if $Y$ is a Choquet space.
\end{itemize}
\end{lemma}
\begin{proof}
Let us start with $(a)$. A meager space can be written as the union of $\omega$ closed nowhere dense subsets. Since $f$ is closed irreducible, the image of a closed nowhere dense subset of $X$ is also closed nowhere dense by \cite[Lemma 2.1]{gates}. The other implication is easier.

Now we prove $(b)$. First assume that $X$ is a Choquet space. We will now use $II$'s strategy on $X$ to produce one on $Y$. Every time player $I$ plays an open set $U_n\subset Y$, let $W_n=f\sp\leftarrow[U_n]$. Using the strategy of Player $II$ on $X$, we obtain an open subset $V_n\subset W_n$ of $X$. Since $f$ is irreducible, Player $II$ plays the non-empty open subset $f\sp\sharp[V_n]$ of $Y$. Since $II$ wins in $X$, there exists $p\in\bigcap\{V_n:n<\omega\}=\bigcap\{W_n:n<\omega\}$ so $f(p)\in\bigcap\{U_n:n<\omega\}=\bigcap\{f\sp\sharp[V_n]:n<\omega\}$. Thus, $II$ also wins in $Y$.

Now assume that $Y$ is a Choquet space, again we transfer $II$'s strategy to $X$. If $I$ plays an open set $U_n\subset X$, consider $W_n=f\sp\sharp[U_n]$ which is open and non-empty in $Y$. Using $II$'s strategy, we obtain an open subset $V_n\subset W_n$. Then Player $II$ plays $f\sp\leftarrow[V_n]\subset U_n$ on $X$. We know that $II$ wins on $Y$ so there exists $p\in\bigcap\{V_n:n<\omega\}$, clearly $f\sp\leftarrow(p)\subset\bigcap\{f\sp\leftarrow[V_n]:n<\omega\}$ so $II$ wins on $X$ as well. This completes the proof of $(b)$.
\end{proof}

\begin{lemma}\label{gameGdense}
Let $X$ be a space and $Y\subset X$ be $G_\delta$-dense in $X$. Then
\begin{itemize}
\item[$(a)$] $X$ is meager if and only if $Y$ is meager, and
\item[$(b)$] $X$ is a Choquet space if and only if $Y$ is a Choquet space.
\end{itemize}
\end{lemma}
\begin{proof}
Start with $(a)$. If $X$ is meager, then $Y$ is also meager because it is dense in $X$. Assume that $Y$ is meager, so $Y=\bigcup\{Y_n:n<\omega\}$ where $Y_n$ is nowhere dense for each $n<\omega$. Let $X_n=\cl{Y_n}$, this is a closed and nowhere dense subset of $X$ for each $n<\omega$. Since $X-\bigcup\{X_n:n<\omega\}$ is a subset of type $G_\delta$ of $X$ that does not intersect $Y$, it must be empty. Thus, $X$ is meager.

Now we prove $(b)$. First assume that $X$ is a Choquet space. As in the proof of Lemma \ref{gamecoabsolute}, we transfer strategies. If $I$ chooses an open subset $U_n$ of $Y$, let $U_n\sp\prime$ be an open subset of $X$ such that $U_n\sp\prime\cap Y=U_n$. Player $II$'s strategy gives an open subset $V_n\subset U_n\sp\prime$ of $X$ so Player $II$ plays $V_n\cap Y\neq\emptyset$. By $II$'s strategy in $X$, we know that $G=\bigcap\{V_n:n<\omega\}$ is a non-empty subset of $X$ of type $G_\delta$. Thus, $G\cap Y\neq\emptyset$, which implies that the described strategy for $II$ is a winning strategy in $Y$.

Now assume that $Y$ is a Choquet space. We again transfer $II$'s strategy in $Y$ to $X$. Every time $I$ chooses an open subset $U_n$ of $X$, consider the open subset $W_n=U_n\cap Y$ of $Y$. The strategy in $Y$ for player $II$ gives an open subset $V_n\subset W_n$. Choose an open subset $V_n\sp\prime$ of $X$ such that $V_n\sp\prime\cap Y=V_n$. So $II$ plays $V_n\sp\prime$. Since $\bigcap\{V_n:n<\omega\}$ is non-empty, we obtain that $\bigcap\{V_n\sp\prime:n<\omega\}$ is non-empty as well. This shows that $II$ has a winning strategy so $X$ is a Choquet space. We have finished the proof.
\end{proof}

\begin{propo}\label{thmgames}
Let $X$ and $Y$ be two separable, nowhere locally compact metrizable spaces such that $\R{X}$ is homeomorphic to $\R{Y}$. Then
\begin{itemize}
\item[$(i)$] $X$ is a Choquet space if and only if $Y$ is a Choquet space, and
\item[$(ii)$] $X$ is meager if and only if $Y$ is meager.
\end{itemize}
\end{propo}
\begin{proof}
Let $K$ be a metrizable compactification of $X$ and $T$ be a metrizable compactification of $Y$, and assume that $X\subset K$ and $Y\subset T$. Let $f:\beta X\to K$ and $g:\beta Y\to T$ continuously extend the identity function. Then $f\restriction_{X\sp\ast}:X\sp\ast\to K-X$ and $g\restriction_{Y\sp\ast}:Y\sp\ast\to T-Y$ are easily seen to be irreducible continuous functions. Notice that if we give $T$ and $K$ some metric, these spaces are the completion of $X$ and $Y$, respectively, with respect to appropriate restrictions of these metrics. Thus, we can use Lemma \ref{choquetcomeager}.

Let us prove $(i)$. Assume that $X$ is a Choquet space. Then, by Lemma \ref{choquetcomeager} $X$ is comeager in $K$. Thus, $K-X$ is meager. By Lemma \ref{gamecoabsolute} applied to $f\restriction_{X\sp\ast}$, $X\sp\ast$ is meager. Recall $\R{X}$ is $G_\delta$-dense in $X\sp\ast$ ($(b)$ in Proposition \ref{structuremetrizable}) so by Lemma \ref{gameGdense}, $\R{X}$ is meager. Thus, $\R{Y}$ is meager. Therefore, we can again use Lemmas \ref{gameGdense} and \ref{gamecoabsolute} to obtain that $Y\sp\ast$ and $T-Y$ are meager as well. Again, Lemma \ref{choquetcomeager} proves that $Y$ is a Choquet space. Using a similar argument it is easy to prove $(ii)$.
\end{proof}

One good hope to extend the results given above is to consider the \emph{strong Choquet game}. This game is similar to the Choquet game, with the exception that $I$ also chooses a point $x_n\in U_n$ and $II$ is required to play so that $x_n\in V_n$. A \emph{strong Choquet space} is one in which $II$ has a winning strategy. See the details in \cite[8D]{kechris}. The important point of this game is that a separable metrizable space is strong Choquet if and only if it is completely metrizable (\cite[8.17]{kechris}). However it is not easy to handle the points $x_n$ in the game to produce analogous results to Lemmas \ref{gamecoabsolute} and \ref{gameGdense}. However, it is possible to prove the following by transfering the strong Choquet game.

\begin{propo}
$\R{\Q}$ is a strong Choquet space.
\end{propo}

So the following remains unanswered.

\begin{ques}\label{quesrationals}
Let $X$ be a metrizable space such that $\R{X}$ is homeomorphic to $\R{\Q}$. Is $X$ $\sigma$-compact?
\end{ques}

Now we make some comments about the use of Proposition \ref{thmgames}. Let $X$ and $Y$ be separable, nowhere locally compact and metrizable. If, for example, $X$ has some open subset that is meager and $Y$ is comeager then we can say that $\R{X}$ and $\R{Y}$ are not homeomorphic using Propositions \ref{subspaces} and \ref{thmgames}. However we are not able to distinguish between, for example, $\R{\Q}$ and $\R{\baire{\omega}\times\Q}$ or between $\R{\baire\omega}$ and $\R{{\sp{2}{P}}\cup{\sp{2}{\Q}}}$ where $P=\mathbb{R}-\Q$.

We also know nothing about non-definable sets (of $\cs$). Thus, we finish the paper with the following question.

\begin{ques}
Do there exist two Bernstein sets $X$ and $Y$ such that $\R{X}$ is not homeomorphic to $\R{Y}$?
\end{ques}

\begin{acknowledgement}
The authors would like to thank the referee for finding various inaccuracies in results and for the useful comments made.
\end{acknowledgement}


\begin{thebibliography}{99}

\bibitem{alexury} Alexandroff, P. and Urysohn, P., ``\"Uber nulldimensionale Punktmengen.'' Math. Ann., 98(1), (1928), 6--36.

\bibitem{chaesmith} Chae, S. B.; Smith, J. H., ``Remote points and G-spaces.'' Topology Appl. 11 (1980), no. 3, 243--246.

\bibitem{vd51} van Douwen, E. K., ``Remote points.'' Dissertationes Math. (Rozprawy Mat.) 188 (1981), 45 pp.

\bibitem{vd82} van Douwen, E. K., ``Transfer of Information about $\beta N-N$ via open remainder maps.'' Ill. J. Math.  34 (1990), 769--792.

\bibitem{eng} Engelking, R., ``General topology.'' Translated from the Polish by the author. Second edition. Sigma Series in Pure Mathematics, 6. Heldermann Verlag, Berlin, 1989. viii+529 pp. ISBN: 3-88538-006-4 

\bibitem{gates} Gates, C. L., ``Some structural properties of the set of remote points of a metric space.'' Canad. J. Math. 32 (1980), no. 1, 195--209.

\bibitem{handbookgenmetric} Gruenhage, G., ``Generalized metric spaces.'' Handbook of set-theoretic topology, 423--501, North-Holland, Amsterdam, 1984. 

\bibitem{hodelhandbook} Hodel, R., ``Cardinal functions. I.'' Handbook of set-theoretic topology, 1--61, North-Holland, Amsterdam, 1984.

\bibitem{kechris} Kechris, A. S., ``Classical descriptive set theory.'' Graduate Texts in Mathematics, 156. Springer-Verlag, New York, 1995. xviii+402 pp. ISBN: 0-387-94374-9 

\bibitem{vmwoods} van Mill, J. and Woods, R. G., ``Perfect images of zero-dimensional separable metric spaces.'' Canad. Math. Bull. 25 (1982), no. 1, 41--47.

\bibitem{morita} Morita, K., ``A condition for the metrizability of topological spaces and for $n$-dimensionality.'' Sci. Rep. Tokyo Kyoiku Daigaku. Sect. A., 5 (1955), 33--36.

\bibitem{peters} Peters, T. J., ``Dense homeomorphic subspaces of $X\sp\ast$ and of $(EX)\sp\ast$.'' Proceedings of the 1983 topology conference (Houston, Tex., 1983). Topology Proc. 8 (1983), no. 2, 285--301.

\bibitem{poncoabsolutes} Ponomarev, V. I., ``On spaces co-absolute with metric spaces.'' Soviet Math. Dokl. 7 (1966), 76--79.

\bibitem{porterwoods} Porter, J. R. and Woods, R. G., ``Extensions and absolutes of Hausdorff spaces.'' Springer-Verlag, New York, 1988. xiv+856 pp. ISBN: 0-387-96212-3

\bibitem{stonenonsep} Stone, A. H., ``Non-separable Borel sets.'', Rozprawy Mat., 28, (1962), 41pp.

\bibitem{woodsremote} Woods, R. G., ``Homeomorphic sets of remote points.'' Canad. J. Math. 23, (1971) 495--502.


\end{thebibliography}
\end{document}